\documentclass{amsart}

\title[On The Entropy of Continuous Flows With Uniformly Expansive Points and The Globalness of Shadowable Points With Gaps]
      {On the Entropy of Continuous Flows With Uniformly Expansive Points and The Globalness of Shadowable Points With Gaps }

\author[A. Arbieto and E. Rego]{A. Arbieto and E. Rego}
\address{Instituto de Matem\'atica, Universidade Federal do Rio de Janeiro, P. O. Box 68530, 21945-970 Rio de Janeiro, Brazil.}
\email{arbieto@im.ufrj.br}
\email{elias@im.ufrj.com}
\thanks{MSC code: 37B40. Keywords: topological entropy, expansivity,  shadowing and continuous flows. A. A. was partially supported by CNPq, FAPERJ and PRONEX/DS from Brazil. E. R. was partially supported by CAPES, CNPq and Faperj from Brazil.}

\newtheorem{theorem}{Theorem}
\newtheorem{corollary}[theorem]{Corollary}
\newtheorem*{mainA}{Theorem A}

\newtheorem*{mainB}{Theorem B}
\newtheorem*{mainC}{Theorem C}

\newtheorem{proposition}[theorem]{Proposition}

\theoremstyle{definition}
\newtheorem{definition}[theorem]{Definition}

\newtheorem*{example}{Example}

\newcommand{\T}{\mathbb{T}}

\newcommand{\al} {\alpha}       
\newcommand{\be} {\beta}        
    
\newcommand{\de} {\delta}

\newcommand{\ka} {\kappa}

\newcommand{\si} {\sigma}       \newcommand{\Si}{\Sigma}

\newcommand{\Z}{\mathbb{Z}}

\newcommand{\R}{\mathbb{R}}
\newcommand{\eps}{\varepsilon}

\newcommand{\per}{\operatorname{Per}}

\begin{document}

\begin{abstract}
In this work we study the problem of positiveness of topological entropy for flows using pointwise dynamics. We show that the existence of a non-periodic nonwandering point of an expansive and non-singular flow with shadowing is a sufficient condition to to obtain positive topological entropy. Moreover, we can deal with flows with singularities, showing that the existence of a non-wandering, non-critical, strongly-shadowable, and uniform-
 expansive point implies the existence of a symbolic subshift. Finally, we discuss pointwise versions of some  shadowing-type properties.
\end{abstract}

\maketitle

\section{Introduction}

There are many ways to measure the complexity of a dynamical system.  Topological entropy is one of the most useful ways to do it.  Its positiveness has as consequence some chaotic behaviour. Hyperbolic systems are a rich source of systems with positive topolgical entropy. For instace,  Horseshoes are  examples of  hyperbolic dynamical systems with positive entropy and were defined by S. Smale in \cite{SM} . Actually, a Horseshoe appears if there exists a transversal homoclinic point and it is equivalent to symbolic shift systems with positive entropy. 

If we move from discrete to the continuous time setting, then the topological entropy is defined for a continuous flow in a natural way. A classical example of a flow with  positive entropy is the suspention of the two symbol  shift map. In \cite{bowen} Bowen and Walters defined expansive flows whose singularities are isolated points of the base space. Later, H.B. Keynes and M. Sears proved in \cite{sears} that Bowen-Walters-expansive flows have positive topological entropy if the base space has topological dimension greater than one. 

In this paper we give some sufficient conditions to ensure positiveness of topological entropy using pointwise dynamics.

Some global dynamical properties can be obtained through analogous properties defined in terms of points, this is what we call pointwise dynamics. For instance, transitivity can be obtained by the existence of a point with dense orbit. 

Apart from the topological entropy, a well known dynamical property is the shadowing property.  It says that trajectories which allow errors are approximated by real ones. A pointwise version of shadowing property was studied in \cite{morales}, where C.A. Morales gave the definition of shadowable points for homeomorphisms and  proved that for compact metric spaces the shadowing property is equivalent to all points be shadowable. The pointwise version of shadowing  for flows was defined in \cite{jesus}. 

In \cite{noix} the authors combined pointwise dynamical information about homeomorphisms in order to conclude positiveness of the topological entropy. The techniques developed there motivated us to try to extend them for flows. The main difficulties is that we need to deal with reparametrizations and the possibility of the existence of singularities. Actually, the latter can  obstruct the use of cross sections, needed to control some reparametrizations.

%In this work we deal with expansivity and the shadowing property in the classical and in the pointwise setting.

% Now we adapt these concepts to bring the results to the continuous time setting. The main difficulty to extend these techniques to flows is that in contrast with the discrete time case, if we work with expansivity and shadowing we have to deal with reparametrizations.   
    
%For pontwise expasivity we obtained the following results:

Our first result deals with global dynamics and complements a result due to Moriyasu \cite{moriyasu}.

\begin{mainA}
	Let $\phi_t$ be an expansive and non-singular continuous flow with the the shadowing property.  If $\Omega(\phi)\setminus Per(\phi)\neq\emptyset$ then there exists $Y\subset X$ such that $\phi|_Y$ is conjugated to a suspension of a subshift.
\end{mainA}

Actually, such methods can be used in the pointwise scenario to obtain positiveness of the entropy, as in the next result.

\begin{mainB} 
	Let $\phi$ be a continuous flow without singularities. If there exists a point $p \in\Omega(\phi)\setminus Per(\phi)$  shadowable and uniformly-expansive, then  $h(\phi)>0$.
\end{mainB}

If the flow has singularities, we may not use the same technics. This is because if some singularity is aproximated by regular orbits, the reparametrizations of shadowing properties can distorts time a lot. But if we consider that  $p$ has a stronger form of  shadowing property we can overcome these difficulties. Moreover, we can prove that the flow has a symbolic subsystem.

\begin{mainC}
	Let $\phi$ be a continuous flow. If there exists  $x\in \Omega(\phi)\setminus Crit(\phi)$ a strongly-shadowable and uniformly-expansive point $x\in X$, then there exists a closed invariant set $Y\subset X$ such that $\phi|_Y$ is a conjugated to a supension of a subshift map with positive entropy.
\end{mainC}

Since subshifts carry a lot of well known topological and statistical information, the previous theorem  improves our knowledge about the behavior of $\phi$. 

\begin{example}
	In order to obtain a non-trivial example of flows for which theorems B and $C$ applies, let us consider a linear Anosov diffeomorphism $f$ of $\T^2$. Then we can blow up  the fixed point of $f$ into a closed disc $D$ and extend $f$ to $D$ as the identity map, obtaining a homeomorphism $\overline{f}$. 
If we consider the suspension flow of $\overline{f}$, then the suspension of all non-wandering points away from the disc $D$ are under the hypothesis of Theorem B. The same is not true for the suspension of the points in $D$. 

In addition, if we define a new flow adding a singularity on the orbit of some point in interior of $D$ we are under the hypothesis of Theorem C.
	
\end{example}

Finally, we also discuss the possibility to obtain pointwise versions of other shadowing-type properties allowing gaps. Actually, we prove that if some shadowing-type property allows large spatial gaps, then we cannot extend the definition of shadowable point to this property.

This paper is organized as follow:
 
In section 2 we present some general definitions and results of continuous flow theory and define precisely the concepts of shadowable and strong-shadowable points.

In section 3 we prove Theorem $A$. 

In section 4 we define and discuss pointwise expansivity and uniform-expansive points.

In section 5 we give a proof for theorem B. 

In section 6 we prove theorem C.

In section 7 we  discuss the possibility to extend the definition of shadowable points to other shadowing-type properties. 
 
 %due to P. Das, T. Das and A. Khan in \cite{indianos } and discuss an appropriate way to extend it to continuos flows. 
 
\section{Preliminaries}

In this section we  define the objects we will work with and  give some general results of continuous flows theory.  

\vspace{0.1in}
\emph{Continuous Flows}

Let $(X,d)$ be a compact metric space. A \textit{continuous flow} on $X$ is a continuous map $\phi:\R \times X \to X$ such that 
$\phi(0,x)=x$  and $\phi(t+s,x)=\phi(s,\phi(t,x))$ for every $x\in X$ and $t,s\in \R$. Let $\phi^t:X\to X$ denote the homeomorphism obtained fixing the time $t$ in the flow $\phi$. The orbit of a point $x\in X$ is the set $O(x)=\cup_{t\in\R}\{\phi^t(x)\}$. 

A point $x\in X$ is  periodic for $\phi$ if there exists $T>0$ such that $\phi^{t+T}(x)=\phi^t(x)$ for every $t\in \R$.  The period $\pi(x)$ of $x$ is the infimum of $T\geq0$ such that previous condition is satisfied. If $\pi(x)=0$ say that $x$ is a singularity. We denote $Sing(\phi)$ for the set of singularities of $\phi$, $Per(\phi)$ for the set of periodic points of $\phi$. We say that $x$ is a critical point for $\phi$, if it belongs to the set $Crit(\phi)=Sing(\phi)\cup \per(\phi)$. 

A \textit{reparametrization} of $\phi$ is an increasing homeomorphism $h:\R\to \R$ fixing the origin. Denote $Rep(\phi)$ the set of reparametrizations of $\phi$. Let us denote $$Rep_{\eps}(\phi)=\{h\in Rep(\phi);|\frac{h(x)-h(y)}{x-y} -1|\leq \eps\}.$$

We say that $x\in X$ is a wandering point for $\phi$ if there are a neighborhood $U$ of $x$ and a time $T>0$ such that $\phi^t(x)\notin U$ for every $t\geq T$. We say that $x$ is a non-wandering point for $\phi$ if it is not a wandering point for $\phi$. 
Let $\Omega(\phi)$ denote the set of non-wandering points of $\phi$.

\vspace{0.1in}

\emph{Cross Sections}

A \emph{cross section of time $e$} is a compact set  $T\subset X$ such that  for every $x\in T$, one has $\phi^{[-e,e]}(x)\cap T=\{x\}$. We denote the interior $T^*$ of $T$ for the intersection of $T$ with the interior of $\phi^{[-e,e]}(T)$.

In \cite{bowen} the authors proved that for any non-singular flow we can chose a finite family of cross-sections which "generates" $\phi$. Precisely, we have the following:

\begin{theorem}\cite{bowen}\label{cross}
	There exists $e>0$ such that the following holds: For every $\alpha>0$ there exists a finite family $\mathcal{T}$ of disjoint cross sections of time $e$ and diameter at most $\alpha$ such that $X=\phi^{[-\alpha,0]}(T^+)=\phi^{[0,\alpha]}(T^+)$ where $T^+=\cup\mathcal{T}$.
\end{theorem}

\textit{Remark:} If we put $\be=\sup\{t>0; \forall x\in T^+, \phi^t(x)\notin T^+\}$, then $0<\be<\al$ and once a point $x\in X$ crosses a cross section $T\in \mathcal{T}$, it takes at least time $\be$ to cross another cross section.

In \cite{sears}, the authors slightly improved previous theorem and obtained the following.

\begin{theorem}\cite{sears}\label{cross2}
	There exists $e>0$ such that the following holds: For every $\alpha>0$ there are a finite families $\mathcal{T}=\{T_1,...,T_k\}$ $\mathcal{S}=\{S_1,...,S_k\}$ of disjoint cross sections of time $e$ and diameter at most $\alpha$ such that $S_i\subset  T_i^*$ for $i=1,...,k$ and   $X=\phi^{[-\alpha,0]}(T^+)=\phi^{[0,\alpha]}(T^+)=\phi^{[-\alpha,0]}(S^+)=\phi^{[0,\alpha]}(S^+)$.
\end{theorem}

\vspace{0.1in}

\emph{Topological Entropy}

Let $V$ be a subset of $X$. A set $E\subset V$ is called a $(t,\eps)$-separated set if for every pair of distinct points $x,y\in E$, there exists some $0\leq u\leq t$ such that  $d(\phi^u(x),\phi^u(y))\geq \eps$. We set $s_t(V,\eps)$ as  the maximal cardinality of all $(t,\eps)$-separated subsets $E\subset V$.

The \emph{topological entropy} of $f$ is the following quantity
$$h(f)=\lim\limits_{\eps\to 0}\limsup\limits_{t\to\infty}\frac{1}{t}\log s_t(X,\eps).$$
 
This limit always exists, see \cite{WAL}. 

\vspace {0.1in}

\emph{The shift map and Supension flows}

Let $\mathcal{A}$ be a finite alphabet. We denote $\Si_{\mathcal{A}}=\mathcal{A}^{\Z}$ for the set of bilateral sequences formed by the symbols in $\mathcal{A}$ endowed with the metric $d(s,t)=\Sigma_{i\in \Z}\frac{1}{2^{|i|}}|s_i-t_i|$, this is a compact metric space. We denote $\Sigma_2=\{0,1\}^{\Z}$.

	The map $\si:\Si_{\mathcal{A}}\to\Si_{\mathcal{A}}$ defined by $\si((s_i))=(s_{i+1})$ is called the \textit{shift map}. 

Let $f:X\to X$ be a homeomorphism and let $r:X\to \R^+$ be continuous map. Let us consider the quotient space 
$$X_r=\{(t,x);0\leq t \leq r(x),x\in X\}/(r(x),x)\sim(0,r(x)).$$
 One can define a metric in order to made $X_r$ a compact metric space(See \cite{thomas} for details).

 We define the suspension flow of $f$ with roof $r$ as the flow $\phi_{\sigma}^t:X_r\to X_r$ induced by the translation $T_t(s,x)=(s+t,x)$. It is well known that the  $h(\phi_{\sigma})=\log\#\mathcal{A}$.   

\vspace {0.1in}

\emph{Expansivity}

The first definition of expansivity for continuous flows is due to Bowen and  Walters in \cite{bowen}. For their definition, singular points are topologically isolated. In \cite{komuro} Komuro defined another kind of expansivity whose most remarkable  example is the Lorenz Attractor. Indeed, Komuro-Expansivity allows the existence of singular points accumulated by regular orbits. It is well known that any Bowen-Walters-expansive flow is a Komuro-expasive flow and in absence of singularities the two definitions are equivalent. 
In this work we will always deal with Komuro-expansive flows and we will always call it expansive flows.

We say that  $\phi$  is an \textit{expansive flow} if  for every $\eps>0$ there exists $\de>0$  such that if $d(\phi^t(x),\phi^{h(t)}(y))<\de$ for every $t\in \R$ and some $h\in Rep(\phi)$, then $y=\phi^t(z)$ with $z\in O(x)$ and $|t|<\eps$.

\vspace{0.1in}
 
\emph{Shadowing}

 A sequence $(x_i,t_i)_{i=a}^b$ with $-\infty\leq a < b \leq \infty$ is a $(\delta,T)$-pseudo-orbit for $\phi$, if $t_i\geq T$ and  $d(\phi^{t_i}(x_i),x_{i+1})<\delta$ for every $a\leq i\leq b$ .
 Let $(x_i,t_i)_{i=a}^b$ be a $(\de,T)$-pseudo-orbit. Define $s_i=\sum_{n=0}^{i-1}t_n$ if $i\geq0$ and $s_i=\sum_{n=i}^{-1}t_n$ if $i<0$.

 \begin{itemize}
 	\item We say that $(x_i,t_i)_{i=a}^b$ is $\eps$-shadowed if there exists a point $x$ and  $h\in Rep(\phi)$ such that $d(\phi^{h(t)}(x), \phi^{t-s_i}(x_i))<\eps$ for every $i$ and $s_{i}\leq t \leq s_{i+1}$.
 	
 	\item We say that $(x_i,t_i)_{i=a}^b$ is $\eps$-strongly-shadowed if there exists a point $x$ and  $h\in Rep_{\eps}(\phi)$ such that $d(\phi^{h(t)}(x), \phi^{t-s_i}(x_i))<\eps$ for every $i$ and $s_{i}\leq t \leq s_{i+1}$.

 \end{itemize}

Now we can define the shadowing properties:
\vspace{0.1in}

We say that $\phi$ has the \textit{shadowing property} if for every $\eps>0$ there exists $\delta>0$ such that any $(\delta,1)$-pseudo-orbit is  $\eps$-shadowed.
\vspace{0.1in} 
  
 We say that $\phi$ has the \textit{strong-shadowing property} if for every $\eps>0$ there exists $\delta>0$ such that any $(\delta,T)$-pseudo-orbit is  $\eps$-strongly-shadowed.

 It is well known that we can replace $(\de,1)$-pseudo-orbit for $(\de,T)$-pseudo orbit for every $T>0$ in the definition of shadowing or strong-shadowing. 

In \cite{komuro2} the M. Komuro proved a relation betwenn shadowing and strong shadowing for flows without singularities.

\begin{theorem}\cite{komuro2}
	Let $\phi$ be a continuous flow without singularities. Then are equivalent:
	\begin{enumerate}
		\item $\phi$ has the shadowing property
		\item For every $\eps>0$, there are $T>0$ and $\de>0$ such that every $\de$-$T$-pseudo orbit of $\phi$ is $\eps$-strongly-shadowed.
	\end{enumerate}
\end{theorem}

 \vspace{0.1in}
\emph{Shadowable Points} 
\vspace{0.1in}

We say that a point $x$ is a \textit{shadowable point} of $\phi$ if for every $\eps>0$ there is a $\de>0$ such that every $(\de,1)$-pseudo-orbit $(x_i,t_i)$ satisfying $x=x_0$ is $\eps$-shadowed. 

\vspace{0.1in}
We say that a point $x$ is a \textit{strongly-shadowable point} of $\phi$ if for every $\eps>0$ there is a $\de>0$ such that every $(\de,1)$-pseudo-orbit $(x_i,t_i)$ satisfying $x=x_0$ is $\eps$-strongly-shadowed.

In \cite{jesus} the authors showed that shadowing property for flows is equivalent to all points be shadowable. It is obvious that the same can be done for strongly-shadowable points. 

\vspace{0.1in}

\section{Proof of Theorem A}

Now we begin to prove theorem A.
We begin choosing $p\in \Omega(\phi)\setminus Per(\phi)$ and fix $\xi>0$ as in Lemma \ref{cross}. For this $\xi$ let $8e>0$ be the expansivity constant of $\phi$. Now let $\{T_1,...,T_i\}$ and $\{S_1,...,S_k\}$ be two families of cross sections of time $\xi$ satisfying the following conditions:
\begin{enumerate}
	\item $S_i\subset T_i^*$ for $i=1,2,...,K$
	\item $diam(T_i)<e$ for $i=1,2,..,K$
	\item $X=\phi^{[-e,0]}(T^+)=\phi^{[0,e]}(T^+)=\phi^{[-e,0]}(S^+)=\phi^{[0,e]}(S^+)$ 
\end{enumerate}

Consider $\be=\sup\{t>0;x\in T^+ \Rightarrow\phi^t(x)\notin T^+ \}$ and $0<2\rho<\be$. Let us define the natural projection $P_i:\phi^{[-\rho,\rho]}(T_i)\to T_i$ by $P_i(y)=\phi^t(y)\in T_i$ with $|t|<\rho$. Since each $S_i$ is compact, we can choose $0<2\eps<e$ such that $B_{\eps}(S_i)\in \phi^{[-\rho,\rho]}(T_i^*)$ for $i=1,2,..,k$.

For $\eps>0$ let $T_1>0$ and $\de_1>0$ such that every $\de_1$-$T_1$-pseudo orbit is $\eps$-strongly-shadowed. 

We will use the shadowing property twice in this proof, so let $0<2\de_2<\de_1$ and $T_2>0$ be such that every $\de_2$-$T_2$-pseudo orbit  is $\de_1$-strongly-shadowable.
Let $0<2\eta<\de_2$ be such that if $d(x,y)<\eta$ then $d(\phi^t(x),\phi)^t(y))<\de_2$ for any $x,y\in X$ and $t\in [-T_2,T_2]$. 

Since $p$ is a non-wandering point, there are $x_a\in B_{\eta}(p)$ and $t_a>T_2$ such that $\phi^{t_a}(x_a)\in B_{\de}(p)$. 
Next, we define the following set: 
\begin{equation*}
A=\{...,(p,T_2),(\phi^{T_2}(x_a),t_a-T_2),(p,T_2),(\phi^{T_2}(x_a),t_a-T_2),...\}. 
\end{equation*}

We claim that $A$ is a $\de_2$-$T_2$-pseudo-orbit. Indeed, by the choice of $\eta$ we have $d(\phi^{T_2}(p),\phi^{T_2}(x_a))<\de_2$ and $d(\phi^{t_a}(x_a),p)<\eta<\de_2$

The shadowing property  implies the existence of a point $a\in B_{\de_1}(b)$ which $\de_1$-shadows $A$. More precisely, there exists a reparametrization $h\in Rep_{\de_1}(\phi)$  such that:
\begin{itemize}
	\item  If $t\in[it_0,it_0+1]$, then $d(\phi^{h(t)}(a),\phi^{t'}(p))\leq \de_1$ with $t'=t-it_0$.  
	\item If $t\in [it_0+1,(i+1)t_0]$, then $d(\phi^{h(t)}(a),\phi^{t'}(x_a)\leq \de_1$ with $t'=t-(i+1)t_0$.
\end{itemize}

Notice that $d(\phi^{h(t)}(a),\phi^{h(t+t_a)}(a))<2\de_1<e$ for $t\in \R$. Then  expansivity implies $\phi^{h(t_a)}\in O(a)$. Thus $a$ is a periodic point and it must be different from $p$.

Set $\eps'=d(p,O(a))$ and let $0<2\de_3<\eps'$ and $T_3>T_2$ be given by the $\eps'$-strong-shadowing . %Fix $0<2\eta'<\de_3$ be such if $d(x,y)<\eta'$ then $d(\phi^t(x),\phi^t(y))<\de_3$ for every $x,y\in X$ and every $t\in[-T_3,T_3]$. 

Now, let $T>T_3$ be such that $d(\phi^T(a),\phi^T(p))>8e$ by expansivity. 

Let $0<2\eta'<\de_3$ be such that $d(\phi^t(x),\phi^t(y))<\de_3$ if $d(y,z)<\eta'$ and $t\in [-T,T]$.

Since $x$ is a non-wandering point, we can choose $x_b\in B_{\eta'}(p)$  such that there exists $t_b>T$ satisfying $\phi^{t_b}\in B_{\eta'}(p)$.

If we repeat the steps to construct $a$, we can find a periodic point $b\in B_{\eps}(x)$ different from $x$ and $a$  which satisfies $d(\phi^T(a),\phi^T(b))>7e$.  We notice that  $d(a,b)\leq d(a,p)+d(p,b)\leq \eps$.

Let $\pi(a)$ and $\pi(b)$ be the períods of $ a$ and $b$, respectively.  

For each $s\in\Sigma_2=\{0,1\}^{\Z}$ we define the sequence $\{(x_i,t_i)\}_{i\in\Z}$ putting $(x_i,t_i)=(a,\pi(a))$ if $s_i=0$ and $(x_i,t_i)=(b,\pi(b))$ if $s_i=1$.

It is easy to see that each $A_s$ is an $\de_1$-$T_1$-pseudo orbit  and therefore there exists a point $y_s$ which $\eps$-shadows $A_s$. Moreover, each shadow is unique by expansivity. Let us define $W=\cup y_s$ with ${s\in\Sigma_2}$ and $Y=\cup_{t\in\R}\phi^t(W)$

\vspace{0.1in}
\textit{Claim:} $Y$ is a compact.
\vspace{0.1in}

To show that $Y$ is closed suppose that $y_n\to y$ with $y_n\in Y$. For each $y_n$ there exists a sequence $s_n\in \Sigma_2$  such that $y_n$ $\eps$-shadows $A_{s_n}$.  To prove our assertion, we have to obtain $s\in \Sigma_2$ and  $h\in Rep_{\eps}(\phi)$ such that $y$ $\eps$-shadows $A_s$. Since $\Sigma_2$ is compact, we can assume that $s_n\to s$. Moreover,  each reparametrization of $y_n$ belongs to $Rep_{\eps}(\phi)$, then $\{h_n\}$ is an equicontinuous sequence. Thus we can assume that $h_n\to h\in Rep_{\eps}(\phi)$.

Fix $T>0$ and let $\rho>0$ such that $d(x,y)<\rho$ implies $d(\phi^t(x),\phi^t(y))<\eps$. Let us  take $y_{n}\in B_{\rho}(x)$. Thus  the chioce of $\rho$ implies $d(\phi^{h(t)}(y),\phi^{h_n(t)}(y_n))<\eps$, if $n$ is large enough. This inequality combined with the fact that the sequences  $s_n$ have the first entries equal to the first entries of $s$ if $n$ is large, gives us that $y$ $\eps$-shadows $A_s$ until time $T$. Now, a straigthforward limit calculation proves that $y$ $\eps$-strongly-shadows $A_s$ and therefore $Y$ is closed. 

\vspace{0.1in}

Since $Y$ is closed we can take new families of cross sections $\mathcal{T}\{T_1',...,T_{k'}'\}$ and $\mathcal{S}=\{S_1',...,S_{k'}'\}$ where $T_i'=T_i\cap Y$ and $S_i'=S_i\cap Y$. Notice that if we define ${T}_i^*{'}=T_i^*\cap Y$  then these families satisfies the properties $(1),(2)$ and $(3)$ of the original families. 

Consider the point $a$ and let $t_0^a$ be the smallest  $t\geq0$ such that $\phi^t(a)\in {S'}^+$ and consider the pair $(S_0^a,t_0^a)$ such that $\phi^{t_0^a}(a)\in S_0^a$ with $S_0^a=\{S_1',...,S_{k'}'\}$. Then define in the same manner the pair $(S_i^a,t_i^a)$ where  is the $i$-th smallest positive time such that $\phi^t(a)\in {S'}^+$. For negative time one can define  these pairs in an analogous way, but using the greatest negative times.

Notice that $\{(S_i^a,t_i^a)\}_{i\in \Z}$ codify the the order and the times in which $a$ intersect the cross sections in the family $\mathcal{S}=\{S_1',...,S_{k'}'\}$. 
Analogously, we can obtain a similar sequence $\{(S_i^b,t_i^b)\}_{i\in \Z}$ for $b$. Since the orbits $a$ and $b$ are periodic,  previous sequences are also periodic. To be more specific, for $z=a,b$ there exists $k_z$ such that $S_{i+k_z}^z=S_i^z$ and $t_{i+k_z}^z=t_i^z+t_{imod(k_z)}^z$. 

Now, since the orbit of any point $y\in Y$ $\eps$-shadows the pseudo-orbits  obtained concatenating the orbits of $a$ and $b$, then by the choice of $\eps$, the analogous sequence $\{(S_i^y,t_i^y)\}$ defined for $y$ is obtained concatenating  the sequences $(S_0^a,...,S_{k_a}^a)$ and $(S_0^b,...,S_{k_b}^b)$ in the first entry, and with times close to the times of crossing for $a$ and $b$. 

In \cite{BW}, Bowen and Walters showed that we can construct a suspension of a subshift of $\Sigma_{\mathcal{T}}$  for which $\psi=\phi|_Y$ is a factor. 
To conclude our proof we briefly will describe the construction used in their proof and discuss the main difference with our case. (For more precise details see \cite{bowen})

The family $\mathcal{S}$ codify the path described by  orbits of $\psi$ in the following way: Since any point $y\in Y$ needs to cross some $S_0^y\in\mathcal{S}$ at most in time $\xi$, we construct for $y$ the sequence $\{(S_i^y,t_i^y)\}$ as above.  
The desired subshift is defined  as the set of $\Sigma_{\psi}$ of the sequences $s$ of $\Sigma_{\mathcal{S}}$ for which there exists a point $y_s$ which crosses the cross sections of $\mathcal{S}$ in the order defined by $s$ and the roof function $r$ of the suspension is given at each point as the time spent to cross the cross sections.
Then the factor map identify orbits points of $\Sigma_{\psi}^r$ orbit of with of points in $Y$ in the obvious way. In their case, there is not a reason for two different orbits of  $\Sigma_{\psi}^r$ be related to different orbits of $Y$.
The same does not occur in our case. Indeed, by the choice of $\eps$ the sequences  $\{(S_i^y,t_i^y)\}$ are in a one-to-one correspondence with the orbits of $\psi$.  So the factor map defined in \cite{BW} is a homeomorphism in our case.

\section{Pointwise Expansivity}

Now we begin to investigate pointwise aspects of expansive flows. The first definition of  pointwise expansivity is due to Reddy in \cite{RD} in the setting of homeomorphisms. He required for any point $x$ in the base space the existence of a positive number $c(x)$ such that the dynamic ball centered at $x$ and with radius $c(x)$  contains only $x$.

Then we can try to adapt this definition to the time continuous case.

\begin{definition}
	A point $x$ is called an expansive point of $\phi$ if for every $\eps>0$ there exists $c(x)>0$ such that if there exists $y$ and a reparametrization $h$ such that $d(\phi^t(x),\phi^{h(t)}(y))<c(x)$ for every $t\in \R$,
	then $ y=\phi^t(z)$ with $z\in O(x)$ and $|t|\leq \eps$. 
	 A flow $\phi$ is pointwise expansive if every point of $x$ is expansive.
\end{definition}
 
 Notice that if $\phi$ is pointwise expansive and  given $\eps>0$, there exist $c$ such that $c(x)\geq c$ for every $x$, then $\phi$ is expansive.
 
 As in the homeomorphism case, $Sing(\phi)$ is finite if $\phi$ is pointwise expansive. Indeed, if there are infinitely many exapansive points then they must accumulate in some point $x$. Then $c(x)$ must be $0$ and  $\phi$ cannot be pointwise expansive.
 
 Another difficult about this definition is to relate pointwise expansivity with expansivity. As in the homeomorphism case, pointwise expansiveness does not implies expansiveness. The following example shows this fact.
 
 \begin{example}
 	In \cite{BW} B. Carvalho and W. Cordeiro give an example of $2$-expansive homeomorphism $f$ with the shadowing property which is not expansive. In their example all the points are expansive points. If one consider a suspension flow of $f$, one will obtain an example of non-expansive flow whose points are expansive.
 \end{example}

 In order to define a pointwise kind of expansivity which allows us to recover  expansivity from the pointwise one, we proceed in the same manner as in \cite{noix} and define uniform-expansive points.

 \begin{definition}
 	We say that $x$ is an uniformly-expansive point of $\phi$ if there is a neighborhood $U$ of $x$ with the following property: For every $\eps>0$ there exists $\de>0$ such that if $y,z\in U$ and there exists a reparametrization $h$ satisfying $d(\phi^t(y),\phi^{h(t)}(z))<\de$ for every $t\in \R$ then $y=\phi^t(w)$ with $w\in O(z)$ and $|t|<\eps$.	
 	
 \end{definition}

We denote $Exp(\phi)$ by the set of uniformly-expansive points of $\phi$.

\begin{proposition}
	$Exp(\phi)$ is an invariant set.
\end{proposition}

\begin{proof}
	Consider $x\in Exp(\phi)$ and let $U$ be its expansivity neighborhood. Fix $\eps>0$ and let $\de$ be given by the expansivity. If $t\in \R$, then $\phi^t(U)$ is a neighborhood of $\phi^t(x)$. Suppose  there are $z,y\in \phi^t(U)$ such that $d(\phi^t(z),\phi^{h(t)}(y))<\de$ for every $t\in \R$. Then $\phi^{-t}(z),\phi^{-t}(y)\in U$ and the previous estimate is still valid. Thus $y=\phi^s(z')$ with $z\in O(z)$ and $|s|<\eps$. Thus $\phi^t(x)$ is an uniformly-expansive point of $\phi$ with expansivity neighborhood $\phi^t(U)$.  
\end{proof}

Now  we use uniformly-expansive points to obtain the expansivity of $\phi$.

\begin{proposition}  $\phi$ is Bowen-expansive if, and only if , every point is uniformly-expansive. 
\end{proposition}

\begin{proof}
	Obviously, we just need to prove the converse.
	Suppose that every point of $X$ is uniformly-expansive and fix $\eps>0$. Since $X$ is compact, we can cover $X$ with a finite number of open sets $U_{x_1},...,U_{x_n}$ given by the expansiveness of the points $x_1,...,x_n$. Set $\de=\min\{\de_{x_1},...,\de_{x_n},\eta\}$, where $\de(x_i)$ is the expansivity consant of $x_i$ and $\eta$ is the Lebesgue number of the cover. If there are points $x,y\in X$ and a reparametrization $h$ such that $d(\phi^t(x),\phi^{\rho(t)}(y))<\de$ for every $t$. Then $x,y\in U_{x_i}$ for some $i$ and therefore $y=\phi^t(z)$ with $z\in O(x)$ and $|t|<\eps$. Then we conclude that $\phi$ is expansive.
\end{proof}

\section{Proof of Theorem B}

Now we are able to prove theorem B. We shall divide the proof in two steps.
First we construct a  $\phi$-invariant set $Y\subset X$  in a very similar way as in theorem A, but we need to pay attention with the initial points of the pseudo-orbits. So we use this set to estimate the entropy of $\phi$.

\vspace{0.1in}
\textbf{Step 1: Constructing $Y$}
\vspace{0.1in} 

Let $p$ be under the hypothesis of the theorem and let $U$ be its expansivity neighborhood. Fix $0<\ka<1$ and let $8e>0$ be the expansivity constant of $p$ such that $B_{8e}(p)\subset U$.

 Let $0<2\eps<e$ be given by the $e$-shadowing though $p$. In \cite{jesus} the authors proved that any point in $B_{\eps}(p)$ is  $2e$-shadowable.
 We will use the shadowing property twice in this proof, so let $0<2\de<\eps$ be such that every $\de$-pseudo orbit trough $p$ is $\eps$-shadowable.
 Let $0<2\eta<\de$ be such that if $d(x,y)<\eta$ then $d(\phi^t(x),\phi^t(y))<\de$ for any $x,y\in X$ and $t\in [-1,1]$. 
 
 Since $p$ is a non-wandering point, there are $x_a\in B_{\eta}(p)$ and $t_a>1$ such that $\phi^{t_0}(x_a)\in B_{\de}(p)$. 
Next, we define the following set: 
\begin{equation*}
	A=\{...,(p,1),(\phi^{1}(x_a),t_a-1),(p,1),(\phi^1(x_a),t_a-1),...\} .
\end{equation*}

We claim that $A$ is a $\de$-$1$-pseudo-orbit through $x$. Indeed, by the choice of $\eta$ we have $d(\phi^1(p),\phi^1(x_a))<\de$ and $d(\phi^{t_a}(x_a),p)<\eta<\de$

The shadowing property through $p$ implies the existence of a point $a\in B_{\eps}(b)$ which $\eps$-shadows $A$. More precisely, there exists a reparametrization $h\in Rep(\phi)$  such that:
\begin{itemize}
	\item  If $t\in[it_0,it_0+1]$, then $d(\phi^{h(t)}(a),\phi^{t'}(p))\leq \de$ with $t'=t-it_0$.  
	\item If $t\in [it_0+1,(i+1)t_0]$, then $d(\phi^{h(t)}(a),\phi^{t'}(x_a)\leq \de$ with $t'=t-(i+1)t_0$.
\end{itemize}

 Notice that $d(\phi^{h(t)}(a),\phi^{h(t+t_a)}(a))<2\eps<e$ for $t\in \R$.  Thus  expansitivty implies that $a$ is a periodic point and it must be different from $p$.

Now, let $T$ be such that $d(\phi^T(a),\phi^T(p))>8e$. Set $\eps'=d(p,O(a))$ and let $0<2\de'<\eps$ be given by the $\eps'$-shadowing through $p$. Fix $0<2\eta'<\de'$ be such if $d(x,y)<\eta'$ then $d(\phi^t(x),\phi^t(y))<\de'$ for every $x,y\in X$ and every $t\in[-1,1]$.

Let $0<\eta'<\eta$ be such that $d(\phi^t(x),\phi^t(y))<\de$ if $d(y,z)<\eta'$ and $t\in [-T,T]$.

Since $x$ is a non-wandering point, we can choose $x_b\in B_{\eta'}(p)$  such that there exists $t_b>T$ satisfying $\phi^{t_b}\in B_{\eta'}(p)$.

 If we repeat the steps to construct $a$, we can find a periodic point $b\in B_{\eps}(x)$ different form $x$ and $a$  satisfying $d(\phi^T(a),\phi^T(b))>7e$.  We notice that  $d(a,b)\leq d(a,p)+d(p,b)\leq \eps$.

Let $\pi(a)$ and $\pi(b)$ be the períods of $ a$ and $b$, respectively. 

 For each $s\in\Sigma_2=\{0,1\}^{\Z}$ we define the sequence $\{(x_i,t_i)\}_{i\in\Z}$ putting $(x_i,t_i)=(a,\pi(a))$ if $s_i=0$ and $(x_i,t_i)=(b,\pi(b))$ if $s_i-1$.

It is easy to see that each $A_s$ is an $\eps$-$1$-pseudo orbit trough $a$ or $b$ and therefore there exists a point $y_s$ which $2e$-shadows $A_s$. Moreover, each shadow is unique by expansivity. Let us define $W=\cup_{s\in\Sigma_2}y_s$ and $Y=\cup_{t\in\R}\phi^t(W)$

It is easy to see $Y$ is an invariant subset and $\phi|_{Y}$ is expansive.

\vspace{0.1in}
\vspace{0.1in}
\vspace{0.1in}
\vspace{0.1in}
\textbf{Step 2: Estimating the Entropy of $\phi$}
\vspace{0.1in}

We begin considering the set $\overline{Y}$ which is a compact $\phi$-invariant set. Hence by Theorem \ref{cross} we can obtain a finite family of cross sections $\mathcal{T}$ of diameter at most $e$ which ``generates" $\overline{Y}$ and $\be>0$ as in the remark of the Theorem \ref{cross}.       
 
 Let us make some considerations on the periodic pseudo-orbits $A_s$. By a periodic pseudo-orbits of period $i$, we mean the pseudo-orbits $A_s$ for which there exists $i\geq0$ such that $s_{n+i}=s_n$ for every $n\in \Z$.     
Notice that the same argument used to prove that $a$ and $b$ are periodic points in the construction of $Y$ can be used to prove that each $y_s$, $e$-shadow of $A_s$ is periodic, if $A_s$ is periodic. For each $n$ denote $B_n=\{s_1^n,...,s_{2^n}^n\}$ the set of $2^n$ distinct periodic sequences of period $n$ in $\Sigma_2$. We claim that its respective $y_{s_1^n},...,y_{s_{2^n}}^n$ shadows are also distinct. Indeed, let $s,s'\in B_n$ be two distinct sequences. Consider $i_0$ the minimal $i\geq0$ such that $s_i\neq s'_i$. Let $(x_i)$ and $(x'_i)$ be the sequences in $\{O(a),O(b)\}^{\Z}$ corresponding to order in which $y_s$ and $y_{s'}$ shadows the orbits of $a$ and $b$. We have that $x_i= x'_i$ for $0\leq i\leq i_0$ and $x_{i_0}=x'_{i_0}$. Since there exists a point $a'\in O(a)$ and a point in $b'\in O(b)$ which are at least $7e$ apart, then $y_s\neq y_{s'}$. Otherwise $y_s$ must to $2e$-shadow simultaneously $O(a)$ and $O(b)$. Set $t_n=\max\{\pi(y_1^n),...,\pi(y_{2^n}^n)\}$.

\vspace{0.1in}
\textit{Claim: The sequence $t_n\to \infty$ as $n\to \infty$.}
\vspace{0.1in}
 
 Indeed, to prove the claim we consider in $B_n$ a sequence $y_s$ such the its corresponding sequence $s\in \Sigma_2$ satisfies $s_i\neq s_{i+1}$ for every $i\in \Z$. Then when $y_s$ is $2e$-shadowing $O(a)$ it needs to cross some cross section near to $a'$. The same occurs when $y_s$ shadows $O(b)$. Thus $y_s$ spends at least time $\be$ to stop shadowing $O(a)$ and begins to shadow $O(b)$. Then $\pi(y_s)\geq n\be$. Thus $t_n$ converges monotonically to infinity.

 Let us define a reparametrization $h$ setting $h(t_n)=n$, $h(-t_n)=-n$, $h(0)=0$ and mapping linearly the intervals $[t_n,t_{n+1}]$ in $[n,n+1]$. Consider the flow $\psi^t=\phi^{h(t)}$. 
For each $n$ the expansiveness of $\phi$ in $Y$ gives us that the set $B_n=\{y_1^n,...,y_{2^n}^n\}$ is $t_n$-$\al$-separated for if $\alpha$ is small enough. Thus $B_n$ is a $n$-$\al$-separated set of $\psi$.
Hence $$h(\psi)\geq\lim\limits_{\eps\to 0}\lim\limits_{n\to \infty}\frac{1}{n}\log \#B_n=\log2>0.$$ 

Finally, $\psi$ is a time change of $\phi$ and its entropy does not vanishes, then we can conclude the same for $\phi$ (see \cite{inv}).

\section{Proof of Theorem C}

    In order to prove Theorem  C we will proceed in a very similar way as in Theorem B. We begin constructing an invariant set.  
	Let $p$ be a non-critical, non-wandering, shadowable and uniformly-expansive point of $\phi$. Let $U$ be the expansiveness neighborhood of $p$. Since in previous theorem we did not use the fact of $\phi$ be non-singuar in the construction of $Y$ we can do exactly the same construction. So let $Y\subset X$ obtained exactly as in Theorem B. The only difference here is that we use the strong-shadowableness of $p$. Then every reparametrization related to the points $y_s$ are in $Rep_{2e}(\phi)$. We will see that this fact allows us to construct $Y$ ``away" from singularities. 

	\textit{Claim: $Y$ is a closed set.}
	\vspace{0.1in}
	
	To show that $Y$ is closed suppose that $y_n\to y$ with $y_n\in Y$. For each $y_n$ there exists a sequence $s_n\in \Sigma_2$  such that $y_n$ $2e$-shadows $A_{s_n}$.  To prove our assertion, we have to obtain $s\in \Sigma_2$ and  $h\in Rep_{2e}(\phi)$ such that $y$ $e$-shadows $A_s$. Since $\Sigma_2$ is compact, we can assume that $s_n\to s$. Moreover,  each reparametrization of $y_n$ belongs to $Rep_{2e}(\phi)$, then $\{h_n\}$ is an equicontinuous sequence. Thus we can assume that $h_n\to h\in Rep_{2e}(\phi)$.
	
	Fix $T>0$ and let $\rho>0$ such that $d(x,y)<\rho$ implies $d(\phi^t(x),\phi^t(y))<e$. Let us  take $y_{n}\in B_{\rho}(x)$. Thus  the chioce of $\rho$ implies $d(\phi^{h(t)}(y),\phi^{h_n(t)}(y_n))<e$, if $n$ is large enough. This inequality combined with the fact that the sequences  $s_n$ have the first entries equal to the first entries of $s$ if $n$ is large, gives us that $y$ $2e$-shadows $A_s$ until time $T$. Now, a straigthfoward limit calculation proves that $y$ $2e$-strongly-shadows $A_s$ and therefore $Y$ is closed. 
	
	\vspace{0.1in}
	
	First, notice that the previous construction implies that $Y\cap Sing(\phi|_Y)=\emptyset$. Hence $\psi=\phi|_Y$  is an expansive non-singular flow with the shadowing property. Thus theorem A gives us the conjugacy. Furthermore, for $\psi$ we are under the hypothesis of Theorem B. Then we know that $h(\phi)>0$.

 \section{Points with Shadowing-type Properties}

In this section we investigate if the possibility to extend the definition of shadowable points for other shadowing-type properties.  By a shadowing-type property we mean a dynamical property for which we can approximate by real orbits sequences of points with some kind of "orbit-like" behaviour.   
Shadowing property is clearly a shadowing-type property for which the orbit-like behaviour for the sequences is to be a pseudo-orbit. 
As we have seen in previous sections there is a well stablished definition for shadowable points, but is it possible to generalize this definiton for other kinds of shadowing-type properties?  
\vspace{0.1in}

\textit{Specification Property}

\vspace{0.1in}
We begin our this discussion with specification property, motivated by the work  of  P. Das , A.G. Khan and T. Das in \cite{indianos }.  They have defined specification points in the setting of continuous maps and have shown that the existence of such points gives rise to chaotic properties, in particular they obtained positiveness of the entropy.
 
 However, it seems that the definition given by \cite{indianos } is too strong. Since it implies the specification property in a global way as we will see below. First, we give the definition of the specification property.

Let $\{x_i\}_{i\geq 0}$ be a sequence of points in $X$. Let $\{t_i\}_{i\geq 0}$, $\{m_i\}_{i\geq 0}$ be sequences of positive integers. We  call the sequence $\mathcal{S}=\{(x_i,t_i)\}_{i\geq 0}$  an orbit sequence and the sequence  $\mathcal{G}=\{m_i\}_{i\geq 0}$  a time-gap. We say that the pair $(\mathcal{S},\mathcal{G})$ is $\eps$-shadowed by $z\in X$ if
 $$ d(f^{s_i+j}(z),f^j(x_i))<\eps, \textrm{ if } j=0,1,2....,t_i-1  $$ 

where $s_0=0$ and $s_i=\sum_{n=0}^{i-1} (t_n+m_n -1)$.

A continuous map $f:X\to X$ has the specification property if for every $\eps>0$, there is $N>0$ such that  any orbit sequence $\mathcal{S}$  with  time-gap $\mathcal{G}$  satisfying $\min\{m_i\}\geq N$ is $\eps$-traced by some point $z\in X$.

 Now, let us recall the definition of specification points as in definition 4.1 of \cite{indianos }. We say that a sequence $\{x_i\}_{i\geq0}$ of points in $X$ is through $x$ if $x_0=x$
 
%  Then we define specification points
 
 \begin{definition}
 	Let $f:X\to X$ be a be a continuous map. A point $x$ is a specification point of $f$ if for every $\eps>0$ there exists $K>0$ such that any orbit sequence $\mathcal{S}$ through $x$ with time-gap $\mathcal{G}$ satisfying $\min\{m_i\}\geq K$ is $\eps$-shadowed for some point $z\in X$.

 \end{definition}

 %We point out that the existence of single specification point is enough to obtain the specification property for the homeomorphism.
 
 The following result connects these two notions.
 
 \begin{theorem}\label{spc}
 	$f$ has the specification property if, and only if, there is some specification point in $X$.
 	
  \end{theorem}

\begin{proof}
	Let $x$ be a specification point of $f$ and fix $\eps>0$.
	Let $K>0$ be given by the specification of $x$. Now chose orbit sequence $\mathcal{S}$ with time gap $\mathcal{G}$ satisfying $\min\{m_i\}\geq K$. If $x_0=x$, then $\mathcal{S}$ is $\eps$-shadowed by some point. If not, define a new orbit segment $\mathcal{S}'$ with time-gap $\mathcal{G}'$ as follows:

Set $\mathcal{S}'=(\{y_i\},\{t'_i\})$ satisfying $y_0=x$ , $y_i=x_{i-1}$ for $i\geq 1$, $t'_0=0$ and $t'_i=t_{i-1}$ for $i\geq 1$. Set  $\mathcal{G}'=\{m'_i\}$ satisfying $m'_0=k$ and $m'_i=m_{i-1}$ for $i\geq 1$. Therefore $\mathcal{S}'$ is an orbit sequence through $x$ with time-gap $\mathcal{G}'$ satisfying $\min\{m'_i\}\geq K$. By the specification property of $x$, there is some $z\in X$ $\eps$-shadowing $\mathcal{S}'$. Now notice that there exists a minimal  time $n\geq K$ such tha $f^n(z)$ is an $\eps$-shadow for $\mathcal{S}$ and this concludes the proof.

\end{proof}

%The previous proposition shows that specification property is equivalent the existence of a single specification point. Therefore it cannot exists an example of system with a specification but not all point being specification points. Thus if we want to get a definition of specification point for flows it is desired a different definition, because all the results concerning topological mixing, Devaney's chaos, topological entropy and others are well known as consequences of the In other words, this definition does not generates new results.

As a corollary, we recast some results in \cite{indianos }.

\begin{corollary}
	If a continuous map $f$ has a specification point, then $f$ has positive entropy, is chaotic in the sense of Devaney and it is topologically mixing.
\end{corollary}

So, in order to try to exploit the idea of specification points for flows, we first need to deal with the following question.

%This gives rise to the following question:    
 
 \vspace{0.1in}
 \textit{Question:} Is it possible to give a definition of specification point for a continuous map  which generates examples of systems for which the set of specification points is neither empty nor the whole space?    

\vspace{0.1in}

 Specification property allows big spatial gaps between orbit segments, this contrasts with shadowing property where the spatial gaps allowed on pseudo-orbits are controled by the shadowing constants. A careful  reading on the proof of Theorem \ref{spc} reveals that this is exactly the feature which turns specification a global definition. With this spirit we can wonder if the same is true for other shadowing-type properties allowing gaps.

Next we give definitions of various shadowing-type properties and inverstigate its pointwise versions. In what follows $f:X\to X$ denotes a continuous map of a compact metric space $X$.

\vspace{0.1in}

\textit{Gluing Orbit Tracing Property}

\vspace{0.1in}

A continuous map $f$ has the gluing orbit tracing property if for every $\eps>0$, there is $N>0$ such that  any orbit sequence $\mathcal{S}$  with  time-gap $\mathcal{G}$  satisfying $\max_{x_i\in \mathcal{G}} \{x_i\}\leq N$ is $\eps$-traced by some point $z\in X$ (see \cite{Sun}).

\vspace{0.1in}

\textit{Average Shadowing Property}

\vspace{0.1in}
 
A sequence $\{x_i\}$ of points in $X$ is an $\de$-average pseudo orbit if there are $N_{\de}>0$ such that $$ \frac{1}{n}\sum_{i=0}^{n-1}d(x_{i+k+1},f(x_{i+k}))<\de  $$ for every $k\geq0$ and $n\geq N_{\de}$.

We say that a $\de$-average pseudo orbit is $\eps$-average shadowed by $z\in X$ if $$ \limsup_{n\to \infty} \frac{1}{n} \sum_{i=0}^{n-1} d(f^i(x),x_i)<\eps. $$

We say that $f$ has the average shadowing property if for every $\eps>0$, there is $\de>0$ such that any $\de$-average pseudo orbit $\{x_i\}$ is $\eps$-average shadowed by some point $z\in X$ (See \cite{Bl}).

\vspace{0.1in}

\textit{Asymptotic Average Shadowing Property}

\vspace{0.1in}

A sequence  of points $\{x_i\}$ in $X$ is an asymptotic average pseudo orbit if  $$\lim_{n\to\infty} \frac{1}{n}\sum_{i=0}^{n-1}d(f(x_i),x_{i+1})\to 0.  $$

An asymptotic average pseudo orbit $\{x_i\}$  is asymptotically shadowed by $z\in X$ if $$\lim_{n\to\infty}\frac{1}{n} \sum_{i=0}^{n-1} d(x_i,f^i(z))\to 0.$$   

We say that $f$ has the asymptotic average shadowing property if every asymptotic pseudo orbit $\{x_i\}$ is asymptotically shadowed by some point $z\in X$ (See \cite{Gu}).

\vspace{0.1in}

\textit{Ergodic Shadowing Property}

\vspace{0.1in}

Let $\mathcal{S}=\{x_i\}$ be a sequence of points in $X$. Now write $P(\mathcal{S},\de)=\{i\geq 0; d(f(x_i),x_{i+1})\geq \de\}$
and $P_n(\mathcal{S},\de)=\{0\leq i \leq n; d(f(x_i),x_{i+1})\geq \de\}$. We say that $\mathcal{S}$ is an $\de$-ergodic pseudo orbit i $$\lim_{n\to \infty}\frac{\#E_n(\mathcal{S},\de)}{n}\to 0.$$

Now let $z\in X$ and write $S(\mathcal{S},z,\eps)=\{i\geq 0; d(f^i(z),x_{i})\geq \de\}$ and $S_n(\mathcal{S},z,\de)=\{0\leq i\geq n; d(f^i(z),x_{i})\geq \de\}$. We say that a point $z$ $\eps$-ergodic shadows $\mathcal{S}$ if $$\lim_{n\to \infty} \frac{\#S_n(\mathcal{S},z,\eps}{n}\to 0.$$

Finally, we say that $f$ has the ergodic shadowing property if for every $\eps>0$ there is $\de>0$ such that every $\de$-ergodic pseudo orbit $\mathcal{S}$ is $\eps$-ergodic shadowed for some point in $z\in X$ (See \cite{Fak}).

Now we define the pointwise versions of previous definitions.

\begin{definition}
Let $f$ be a continuouos map.
\begin{itemize}
\item $x$ is a point with gluing orbit tracing property if for every $\eps>0$ there is some $K>0$ such that any orbit sequence through $x$ with time gap $\mathcal{G}$ satisfying $\max\{m_i\}\leq K$ is $\eps$-shadowed by some point on $X$.
\item $x$ is an average shadowable point if for every $\eps>0$ there is some $\de>0$ such that any $\de$-average pseudo orbit is $\eps$-average shadowed for some point in $X$.

\item $x$ is an asymptotic average shadowable point if every asymptotic pseudo through $x$ is asymptotic average shadowed by some point in $X$
\item $x$ is an ergodic shadowable point if for every $\eps>0$ there is some $\de>0$ such that every $\de$-ergodic pseudo orbit through $x$ is $\eps$-ergodic shadowed by some point in $X$.

\end{itemize}

\end{definition}

Then we have the following result:

\begin{theorem}
Let $f:X\to X$ be continuouos map.

\begin{enumerate}

\item $f$ has the gluing orbit tracing property if, and only if, there is some gluing orbit tracing point on $X$.
\item $f$ has the average shadowing property if, and only, if there is some average shadowable point in $X$.
\item $f$ has the asymptotic shadowing property if, and only if, there is some asymptotic average shadowable poin in $X$
\item $f$ has the ergodic shadowing property if, and only if, there is some ergodic shadowable point in $X$. 

\end{enumerate}
\end{theorem}

\begin{proof} We just need to prove one direction for each item.

\begin{enumerate}
\item The proof is totally analogous to the prof of Theorem \ref{spc} and we shall omit it.
\item Let $x$ be an average shadowable point and fix some $\eps>0$. Let $\de>0$ be given by the average shadowing through $x$ and fix any $\frac{\de}{2}$-average pseudo orbit $\{x_i\}$ on $X$. Now construct a a new sequence $\{y_i\}$ such that $y_0=x$ and $y_i=x_{i-1}$ for $i\geq 1$. Notice that $\{y_i\}$ is a sequence through $x$. To show that $\{y_i\}$ is a $\de$-average pseudo orbit through $x$ let us consider $N>N_{\de}$ such that $\frac{1}{N}d(f(x),x_0)<\frac{\de}{2}$, where $N_{\de}$ is giver by $\{x_i\}$. This $N$ can be taken since, $X$ is compact and then has finite diameter. 
Now, for $n\geq N$ and  $k\geq0$ we have 

$$\frac{1}{n} \sum_{i=0}^{n-1}d(f(y_{i+k},y_{i+k+1}))\leq\frac{d(f(x),x_0)}{n}+\frac{1}{n}\sum_{i=0}^{n-2}d(f(x_{k+i},x_{k+i+1})\leq \frac{\de}{2}+\frac{\de}{2} $$  

Therefore $\{y_i\}$ is $\eps$-average shadowed by some point $z$ and by construction $f(z)$ is clearly an $\eps$-average shadow for $\{x_i\}$.

\item Let $x$ be an asymptotic average shadowable point for $f$ and let$\{x_i\}$ be any asymptotic average pseudo orbit for $f$. 
Construct a new sequence $\{y_i\}$ such that $y_0=x$ and $y_i=x_{i-1}$ for $i\ge 1$.  Then we have 
$$\lim_{n\to\infty} \sum_{i=0}^{n-1} d(f(y_i),y_{i+1})=\lim_{n\to \infty} \frac{d(f(x),x_0)}+\lim_{n\to\infty} \frac{1}{n} \sum {i=0}^{n-1}d(f(x_i),x_{i+1}))=0$$
 Then $\{y_i\}$ is an asymptotic average pseudo orbit trough $x$ and it is asymptotic average shadowed for some $z\in X$.
Now it is clear that $f(z)$ asymptotic average shadows $\{x_i\}$.

\item Let $x$ be an ergodic shadowable point for $f$ and let$\{x_i\}$ be $\de$-ergodic  pseudo orbit for $f$. 
Construct a new sequence $\{y_i\}$ such that $y_0=x$ and $y_i=x_{i-1}$ for $i\ge 1$.  We claim that $\{y_i\}$ is a $\de$-ergodic pseudo orbit through $x$. Indeed, if $d(f(x),x_0)<\de$ there is nothing to do. If not, we just need to observe that 
$$\lim_{n\to \infty}\frac{\# P_n(\{y_i\},\de)}{n}=\lim_{n\to \infty}\frac{\# P_n(\{x_i\},\de)+1}{n}=0. $$ 

 Then $\{y_i\}$ is $\eps$-ergodic shadowed for some $z\in X$ and it is clear that $f(z)$ $\eps$-ergodic shadows $\{x_i\}$.

\end{enumerate}

\end{proof}
 
\emph{Remark: All the results on this section are easily adapted to the context of homeomorphisms and continuouos flows}

As we have seen,  if a shadowing-type property allows us to have big spatial gaps in the sequences $\{x_i\}$, then we can recover the original property by its pointwise version if we define it analogously to shadowable points. Thus we cannot generate examples with the pointwise property, but without the global property.  This brings us to the following problem: 

 \textit{Question:} Is it possible to give a definition of point with some shadowing type property for a continuous map  which generates examples of systems for which the set of such points point is neither empty nor the whole space?


\begin{thebibliography}{00}

\bibitem{jesus}
Aponte, J. Villavicencio. H
{\em Shadowable Points for Flows}
Journal of Dynamical  and Control Systems, 24, (2018), 701–719


\bibitem{noix}
Arbieto, A. Rego, E.
{\em Positive Entropy Through Pointwise Dynamics} 
To appear in Proceedings of the American Mathematical Society,  https://doi.org/10.1090/proc/14682 



\bibitem{aoki}
Aoki, N. Hirade, K.,
{\em Topological Theory of Dynamical Systems},
Recent Advances, North-Holland Mathematical Library, 52, North-Holland Publishing Co., Amsterdam, 1994

\bibitem{Bl}
Blank, M.L.
{\em Small perturbatious of chaotic dynamical systems.}
 Russian Math. Surveys 44, 1–33 (1989)

\bibitem{bowen}
Bowen, R.
{\em Expansive One-Paramenter Flows}
Journal of Differential Equations, 12, (1972), 180-193

\bibitem{baianos}
Bonfim, T. Torres, M.P. Varandas, P.
{\em Topological Features of Flows with the Reparametrized Gluing Orbit Property}
Journal of Differential Equations, 262 (2017), 4292-4313

\bibitem{BW}
Carvalho, B. Cordeiro, W.,
{\em N-expansive homeomorphisms with the shadowing property},
Journal of Differential Equations, 261, (2016), 3734-3755


\bibitem{indianos }
Das P., Khan A.G., Das T., 
{\em Measure Expansivity and Specification for Pointwise Dynamics}
Bulletin of Brazilian Mathematical Society, (2019)

\bibitem{devaney}
Devaney R.,
{\em An Introduction to Chaotic Dynamical Systems}
Boulder-CO: Westview Press; 2003

\bibitem{Fak}
Fakhari, A., Ghane, F.H.
{\em On shadowing: Ordinary and ergodic.}
 J. Math. Anal. Appl. 364, 151–155 (2010)

\bibitem{fedeli}
Fedeli, A., Le donne A.
{\em A note on the uniform limit of transitive
	dynamical systems}

\bibitem{Gu}
Gu R-B.,
{\em The asymptotic average shadowing property and transitivity.},
 Nonlinear Anal Theory Methods Appl 2007;67:1680–9

\bibitem{kawaguchi}

Kawaguchi, N.
{\em Properties of Shadowable Points: Chaos
	and Equicontinuity}
Bull. Braz. Math. Soc. (N.S.) 48 (2017), no.
4, 599–622].


\bibitem{sears}
Keynes, H. B, Sears, M. 
{\em Real-expansive flows and topological dimension}
Ergodic Theory and Dynamical Systems, Volume 1, Issue 2 June 1981 , pp. 179-195 
 


\bibitem{komuro}
Komuro, M.
{\em Expansive properties of Lorenz attractors} 
Theory of Dynamical Systems and Its Application to Nonlinear Problems, 4–26. World Sci. Publishing, Kyoto.

\bibitem{komuro2}
Komuro, M.
{\em One-parameter flows with the pseudo orbit tracing property}
 Monatshefte für Mathematik, September 1984, Volume 98, Issue 3, pp 219–253


\bibitem{morales2}

Morales C.A., Sirvent V.F.,
{\em Expansive Measures}
Publicações Matemáticas, 29º Colóquio Brasileiro de Matemática, IMPA ,(2013).

 
\bibitem{morales}
Morales C.A.,
{\em Shadowable points}
Dynamical Systems (Print), v. 31, p. 347-356, 2016.


\bibitem{moriyasu}
Moriyasu K.
{\em Continuous flows with Markov families}
Japan. J. Math., Vol. 16, No. 1, 1990



\bibitem{RD}
Reddy, W.
{\em Pointwise Expansion Homeomorphisms}
Journal of London Mathematical Society, (1970), 232-236

\bibitem{ROB}
Clark Robinson,
{\em Dynamical Systems. Stability, Symbolic Dynamics and Chaos},
CRC Press inc., (1995)


\bibitem{SM}
Stephen Smale.,
 {\em Differentiable dynamical systems},
 Bulletin of the American Mathematical Society 73 
 (1967), 747–817

\bibitem{moo}
Subrahmonian Moothathu T.K.,
{\em Implications of pseudo-orbit tracing property for continuous maps on
	compacta},
Topology and its Applications 158 (2011), 2232–2239


\bibitem{Sun}
 Sun P,
{\em  Minimality and gluing orbit property},
 Discrete Contin. Dyn. Syst., Ser. A 39 (7) (2019) 4041–4056.


\bibitem{inv}
Sun W.,  Vargas E.
{\em Entropy of flows, revisited} 
Boletim da Sociedade Brasileira de Matemática (1999), V.30, N.3, 315-333

\bibitem{thomas}
Thomas R. F.,
{\em Stability Properties of One-Parameter Flows}
 Proceedings of the London Mathematical Society, Volume s3-45, Issue 3, November 1982, Pages 479–505


\bibitem{WAL}
Walters, P.
{\em An Introduction to Ergodic Theory},
Graduate Texts in Mathematics, Springer (1982).

\bibitem{EP}
X. Ye, G. Zhang, 
{\em Entropy points and applications}
Trans. Amer. Math. Soc. 359 (12) (2007) 6167–6186.



\end{thebibliography}
\end{document}